\documentclass{amsart}

\usepackage{amsmath,amsfonts,amsthm,amstext,amssymb,amscd,amsbsy}
\usepackage[latin1]{inputenc}
\usepackage[T1]{fontenc}
\usepackage{indentfirst} 
\usepackage[dvips]{color}
\usepackage{epsfig}
\usepackage{color}
\usepackage{graphics}

\newtheorem{definicao}{Definition}

\newtheorem{lema}[definicao]{Lemma}
\newtheorem{obs}[definicao]{Remark}
\newtheorem{corolario}[definicao]{Corollary}
\newcommand{\rn}[1]{\mathbb{R}^{#1}}

\newcommand{\bx}{{\bf x}}
\newcommand{\ze}{{\bf 0}}

\newcommand{\bz}{{\bf z}}

\def\bq{\begin{equation}}
\def\eq{\end{equation}}

\title[Sliding vector field over tori and spheres]{On non-smooth vector fields having a torus or a sphere as the sliding manifold}
\author{Ricardo Miranda Martins}
\address{Instituto de Matem\'atica, Estat\'istica e Computa\c c\~ao Cient\'ifica. Universidade Estadual de Campinas. 13083--859. Campinas, SP, Brazil.}
\email{rmiranda@ime.unicamp.br}
\thanks{The author is supported by FAPESP--BRAZIL grants 2010/13371-9 and 2012/06879-1.}

\begin{document}

\begin{abstract}In this paper we consider a non-smooth vector field $Z=(X,Y)$, where $X,Y$ are linear vector fields in dimension 3 and the discontinuity manifold $\Sigma$ of $Z$ is or the usual embedded torus or the unitary sphere at origin. We suppose that $\Sigma$ is a sliding (stable/unstable) manifold with tangencies, by considering $X,Y$ inelastic over $\Sigma$. In each case, we study the tangencies of the vector field $Z$ with $\Sigma$ and describe the behavior of the trajectories of the sliding vector field over $\Sigma$: they are basically closed.
\end{abstract}

\maketitle
\section{Introduction}

  The theory of the non-smooth dynamical systems has been developing very fast in the last years, mainly due to its strong relation with branches of applied science, such as mechanical and aerospace engineering. Indeed, discontinuous systems are in the boundary between mathematics, physics and engineering. 

In many cases, the non-smooth systems are described by systems of ODEs in such a way that each system is defined in a region of the phase portrait. The boundary of such regions is called the discontinuity manifold \cite{1}.

There are a lot of research being made about the local behavior of non-smooth systems, near the discontinuity manifold \cite{2}. An interesting phenomenon that occurs frequently is the sliding motion, when the trajectories on both sides of the discontinuity manifold slides over the manifold after the meeting, and there remains until reaching the boundary of the sliding region \cite{3,4}. This phenomenon has been studied specially on relay control systems and systems with dry friction \cite{5,6}.

The mathematical formalization of the theory of discontinuous systems was made precise by Filippov \cite{7} and Teixeira-Sotomayor \cite{8,9}, which classified and distinguished three types of regions on the discontinuity manifold (sewing region, sliding region and escape region). These regions fully describes what can happen on such manifold. Teixeira and Sotomayor \cite{9} have introduced the concept of regularization of discontinuous vector fields, which allows to study phenomena that occurs in regular $C^r$ vector fields, but in the context of non-smooth vector fields.

One of the biggest limitations of the theory of non-smooth systems until now is the lack of global results \cite{4}. In particular, the discontinuity manifold is often taken as a hyperplane.

This article presents some results for the case where the discontinuity manifold is a torus or a sphere in the three-dimensional Euclidean space. We consider inelastic vector fields, which are a special class of vector fields with the additional property that the discontinuity manifold is composed only of escape and sliding regions (separated by tangency curves). Inelastic vector fields are widely used in mechanical models \cite{10}. 

This paper is outlined as follows: in Section 2, we introduce the concept of non-smooth vector fields. In Section 3 we discuss inelastic vector fields. In Section 4 we state the main results. In Sections 5 and 6 we proof the main results.
\section{\label{nsvf}Non-smooth vector fields}

In this section we recall some concepts about non-smooth dynamical systems, according to the Filippov convention \cite{7}. We shall adopt $\rn{3}$ as the ambient space, but almost of the definitions and equations can be made for arbitrary dimensions \cite{4}.

Let $h:\rn{3},0\rightarrow\rn{},0$ be a function with $0$ as regular value and consider the manifold $M=h^{-1}(\{0\})$. Also define $M^+=h^{-1}((0,\infty))$ and $M^-=h^{-1}((-\infty,0))$. Put $\Lambda(3)$ the space of the $C^r$-vector fields over $\rn{3},0$, for $r\geq 2$, and $\Omega(3)$ the space of the vector fields $Z$ defined over $\rn{3},0$ as
\[
Z(\bx)=\left\{
\begin{array}{l}
X(\bx), \ {\rm if} \ h(\bx)>0,\\
Y(\bx), \ {\rm if} \ h(\bx)<0,\\
\end{array}
\right.
\]
where $X,Y\in\Lambda(3)$ and $\bx=(x_1,x_2,x_3)$. We denote $Z=(X,Y)$. We can identify $\Omega(3)=\Lambda(3)\times\Lambda(3)$ and inherit a topology from this product.

Now we define $Z$ over $M$. We break $M$ into three (possibly disconnected) regions, according to the following rule (we recall that $Xh(\bx)=X(\bx)\cdot \nabla h(\bx)$):\\
\noindent $M_1=\{\bx\in M; Xh(\bx)\cdot Yh(\bx)>0\}$ is the sweing region;\\
\noindent $M_2=\{\bx\in M; Xh(\bx)>0, \ Yh(\bx)<0\}$ is the escape region;\\
\noindent $M_3=\{\bx\in M; Xh(\bx)<0, \ Yh(\bx)>0\}$ is the sliding region.

If $p\in M$ and $Xh(p)=0$, then $p$ is said to be a tangency point for $X$. If $Xh(p)=0$ and $X^2h(p)\neq 0$ then the tangency is quadratic, and if $Xh(p)=X^2h(p)=0$ but $X^3h(p)\neq 0$ the tangency is cubic.

The trajectories of $Z$ through a point $p$ are defined in the following way:\\
a) If $p\in M^+$, then the trajectory of $Z$ through $p$ is trajectory of $X$ through $p$, until the trajectory reach $M$;\\
b) If $p\in M^-$, then the trajectory of $Z$ through $p$ is trajectory of $Y$ through $p$, until the trajectory reach $M$;\\
c) If $p\in M_1$, then $p$ connect two regular trajectories of $X$ and $Y$; we define the trajectory through $p$ as the concatenation of these trajectories;\\
d) If $p\in M_2\cup M_3$ then the trajectory of $p$ is the trajectory of the sliding vector field associated to $X$ and $Y$, that is defined as the vector field
\bq\label{sliding}
S(X,Y)({\bx})=\dfrac{1}{Yh({\bx})-Xh({\bx})}\left( Yh({\bx})X({\bx})-Xh({\bx})Y({\bx}) \right).
\eq
We remark that the manifold $M$ is invariant by the flow of $S(X,Y)$.\\
e) If $p$ is a tangency point ($p\in \partial M_2\cup\partial M_3$) and there are trajectories of the sliding vector field starting and ending at $p$, then the trajectory through $p$ is just the concatenation of these trajectories (adapted from Definition 2.5 of \cite{2}).

So if the closure of the concatenation between a trajectory of the sliding vector and a trajectory of the escape vector field is transversal do the curve of tangencies, the closure itself is a trajectory.

For $p$ in the boundary of the regions $M_1,M_2,M_3$, the definition of its trajectory is much more complicated (or even undefined yet), and is outside our scopus. We recommend \cite{2} for these definitions.

\section{Inelastic systems}

In this section we establish some general results on inelastic vector fields. 

Let $h:\rn{3},0\rightarrow\rn{},0$ be a function with $0$ as regular value and consider the manifold $M=h^{-1}(\{0\})$. Two vector fields $X,Y$ are inelastic over $M$ if \bq\label{inelastic}(Xh)(\bx)=-(Yh)(\bx)\eq for all $x\in M$.

Inelastic vector fields appears naturally in mechanic models \cite{11} and in reversible non-smooth models. A planar non-smooth vector field $Z=(X,Y)$ with discontinuity set $\Sigma\subset\rn{2}$ is said to be reversible with respect to an involution $\varphi:\rn{2}\rightarrow\rn{2}$ with $Fix(\varphi)\subset\Sigma$, where $Fix(\varphi)=\{\bz\in\rn{2}: \, \varphi(\bz)=\bz\}$ if $\varphi X=-Y\varphi$.  If we take $\varphi,\Sigma$ such that $\Sigma=\{(x,y): \, x=0\}=Fix(\varphi)$ and the vector fields $X,Y$ are even in the second variable, then the pair $(X,Y)$ is inelastic over $\Sigma$.

For $\bx=(x_1,x_2,x_3)$, we consider $\sigma_1(\bx)=x_1^2+x_2^2+x_3^2-1$ and $\sigma_2(\bx)=(x_1^2+x_2^2+x_3^2+3)^2-16(x_1^2+x_2^2)$. Denote $\Sigma_j=\sigma_j^{-1}(\{0\})$, $\Sigma_j^-=\sigma_j^{-1}((-\infty,0))$  and $\Sigma_j^+=\sigma_j^{-1}((0,+\infty))$. We note that $\Sigma_1$ is a sphere and $\Sigma_2$ is a torus.

Now we give a characterization of linear $\Sigma_j$-inelastic vector fields.

\begin{lema}[Linear case]\label{lema1}Let $X,Y$ be linear vector fields in $\rn{3}$. Consider $X(\bx)=A\bx$ and $Y(\bx)=B\bx$, for $A=(a_{i,j})$ and $B=(b_{i,j})$ matrices.\\
a) If $X,Y$ are inelastic over the sphere $\Sigma_1$, then 
\[B=\left( \begin {array}{ccc} -a_{{1,1}}&-a_{{2,1}}-b_{{2,1}}-a_{{1,2}}&
-a_{{1,3}}-a_{{3,1}}-b_{{3,1}}\\ \noalign{\medskip}b_{{2,1}}&-a_{{2,2}
}&-a_{{3,2}}-b_{{3,2}}-a_{{2,3}}\\ \noalign{\medskip}b_{{3,1}}&b_{{3,2
}}&-a_{{3,3}}\end {array} \right) \]
b) If $X,Y$ are inelastic over the torus $\Sigma_2$, then
\[
B= \left( \begin {array}{ccc} -a_{{1,1}}&-a_{{1,2}}-a_{{2,1}}-b_{{2,1}}&
-a_{{1,3}}\\ \noalign{\medskip}b_{{2,1}}&-a_{{2,2}}&-a_{{2,3}}
\\ \noalign{\medskip}-a_{{3,1}}&-a_{{3,2}}&-a_{{3,3}}\end {array}
 \right)
\]

\end{lema}
\begin{proof}The equation \eqref{inelastic} with $h=\sigma_j$ is equivalent to a linear system of equations.

Solving this system for the $b_{i,j}$'s, we obtain the expression of $B$ given in the statements.
\end{proof}

\section{Main results\label{mainresults}}

For $\bx=(x_1,x_2,x_3)$, we consider $\sigma_1(\bx)=x_1^2+x_2^2+x_3^2-1$ and $\sigma_2(\bx)=(x_1^2+x_2^2+x_3^2+3)^2-16(x_1^2+x_2^2)$. Denote $\Sigma_j=\sigma_j^{-1}(\{0\})$, $\Sigma_j^-=\sigma_j^{-1}((-\infty,0))$  and $\Sigma_j^+=\sigma_j^{-1}((0,+\infty))$.

Consider $X,Y$ two linear vector fields defined over an open set $U\subset\rn{3}$ containing $\Sigma_j$, $j=1,2$. Suppose that $X,Y$ are singular at $\ze$. We can write $X(\bx)=A\bx$ and $Y(\bx)=B\bx$, where $A=(a_{i,j})_{3\times 3}$ and $B=(b_{i,j})_{3\times 3}$.

Fix $j=1$ or $j=2$. Suppose that $(X,Y)$ is a $\Sigma_j$-inelastic pair, that is $\left.(X\sigma_j+Y\sigma_j)\right|_{\Sigma_j}=0$, and consider $Z_j=(X,Y)$ the non-smooth vector field
\[Z_j(\bx)=\left\{
\begin{array}{lcl}
X(\bx)&,& \bx\in \Sigma_j^+,\\
Y(\bx)&,& \bx\in \Sigma_j^-.
\end{array}
\right.
\]

Consider the sliding vector field $S_j$ defined over $\Sigma_j^*$
\bq\label{Sj}S_j(\bx)=\dfrac{1}{(Y\sigma_j)(\bx)+(X\sigma_j)(\bx)} \left((Y\sigma_j)(\bx)X(\bx)-(X\sigma_j)(\bx)Y(\bx)\right), \ \ \bx\in\Sigma_j^*,\eq
where $\Sigma_j^*=\{\bx\in\Sigma_j;\, Y\sigma_j(\bx)X\sigma_j(\bx)\neq 0\}$.

\bigskip

\noindent {\bf Theorem A:} If $X,Y$ are linear and inelastic over $\Sigma_1$ (sphere), then the sliding vector field $S_1$ defined on the sphere (if non trivial) has two equilibrium points $p_\pm$, and every other trajectory is closed and equal to the intersection between $\Sigma_1$ and a plane orthogonal to the line $\overline{p_+p_-}$ joining the two equlibria.

\bigskip

\noindent {\bf Theorem B:} If $X,Y$ are linear and inelastic over $\Sigma_2$ (torus) and $X\sigma_2(\bx)$ is quadratic, then the sliding vector field $S_2$ can be defined on $S_2$ except on two circles, that consists of curves of singular points. For all other points in $\Sigma_2$, if non trivial, the trajectory of $S_2$ is closed and planar.

\section{Proof of Theorem A\label{pta}}

Before proving Theorem A, we we derive the equations for the tangency points and the expression of the sliding vector field. We fix $\sigma_1(\bx)=x_1^2+x_2^2+x_3^2-1$, so that $\Sigma_1=\sigma_1^{-1}(\{0\}$ is the sphere.

Let $X(\bx)=A\bx$ and $Y(\bx)=B\bx$ be two linear vector fields, inelastic over $\Sigma_1=\sigma_1^{-1}(\{0\})$. Note that $B$ is given by Lemma \ref{lema1}.

Then
\bq\label{xh-esfera}
\begin{array}{lcl}
X\sigma_1(\bx)&=&2a_{1,1}x_1^{2}+ 2a_{2,2}x_2^{2}+2a_{3,3}x_3^{2}\\
&+&( 2a_{1,2}+2a_{2,1} ) x_1x_2+ ( 2a_{1,3}+2a_{3,1} ) x_1x_3\\
&+&( 2a_{2,3}+2a_{3,2} ) x_2x_3\\
\end{array}
\eq

According to the classification of quadratic surfaces, due to [Beyer 1987], the set $\{\bx\in\rn{3};\, X\sigma_1(\bx)=0\}$ is or the empty set, or a plane, or the union of two intersection planes, or an elliptic cone (the union, by the vertex, of two cones along the same axis). So, when not empty, the intersection of $\{\bx\in\rn{3};\, X\sigma_1(\bx)=0\}$ with the sphere can be or a circle, or two disjoint circles, or two circles intersecting transversally.

Note that if we define $Q=A+A^t$, that is,
\[Q=
\left(
\begin{array}{ccc}
2a_{1,1}&a_{1,2}+a_{2,1}&a_{1,3}+a_{3,1}\\
a_{1,2}+a_{2,1}&2a_{2,2}&a_{2,3}+a_{3,2}\\
a_{1,3}+a_{3,1}&a_{2,3}+a_{3,2}&2a_{3,3}
\end{array}
\right)\]
 then $X\sigma_1(\bx)=\bx^t Q \bx$, so if $Q$ is negative definite, then $Xh$ is the negative of a quadratic form, and $X\sigma_1(\bx)<0$ for all $\bx$. In this case, there are no tangencies. 

\begin{lema}\label{lema-ND}If $A$ is negative definite, then $X\sigma_1(\bx)<0$ and $X\sigma_1(\bx)>0$, $\forall \bx\in S^2$. So there are no tangency points over the sphere.
\end{lema}
\begin{proof}We know by the discussion of the last paragraph that if $Q=A+A^t$ is positive definite, then there are no tangencies.  In \cite{12} is shown that a matrix $L$ is positive (negative) definite if, and only if, $L+L^t$ is positive (negative) definite. This complete the proof.
\end{proof}

\begin{lema}\label{lema-ndm}Suppose that $A$ is a negative definite matrix. Then:\\
a) The sphere is a sliding region, and the sliding vector field $S_2$ defined over the sphere is given by
\[
\begin{array}{lcl}
S_2(\bx)&=&\left({ -\dfrac{a_{2,1}+b_{2,1}}2y-\dfrac{a_{3,1}+b_{3,1}}2 z} \right)\dfrac{\partial}{\partial x_1}\\
&+&\left({ \dfrac{a_{2,1}+b_{2,1}}2x-\dfrac{a_{3,2}+b_{3,2}}2z} \right)\dfrac{\partial}{\partial x_2}\\
&+&\left({ \dfrac{a_{3,1}+b_{3,1}}2x+\dfrac{a_{3,2}+b_{3,2}}2y} \right)\dfrac{\partial}{\partial x_3}\\
\end{array}
\]
\noindent b) The sliding vector field $S_2$ has two symmetric equilibrium points,
\[p_{\pm}=\pm \dfrac{1}{\rho} (  a_{3,2}+b_{3,2},-(a_{3,1}+b_{3,1}),a_{2,1}+b_{2,1}  ),\]
where $\rho=
b_{2,1}^{2}+a_{2,1}^{2}+a_{3,2}^{2}+a_{3,1}^{2}+b_{3,1}^{2}+b_{3,2}^{2}+2b_{2,1}a_{2,1}+ 2b_{3,2}a_{3,2}+2a_{3,1}b_{3,1}$, and all other trajectories of $S_2$ are closed (and contained in the planes orthogonal to $\overline{p_-p_+}$ that intersects the sphere).
\end{lema}
\begin{proof}We compute $S_2$ using formula \eqref{Sj}, and obtain a linear vector field $S_2(\bx)=[S_2]\bx$, where
\[[S_2]=\dfrac12 \left(
\begin {array}{ccc}
0&- a_{{2,1}}-b_{{2,1}}&- b_{{3,1}}- a_{{3,1}}\\
b_{{2,1}}+a_{{2,1}}&0&- a_{{3,2}}-b_{{3,2}}\\
a_{{3,1}}+ b_{{3,1}}&b_{{3,2}}+a_{{3,2}}&0
\end {array}
\right)\]
As the matrix $[S_2]$ is skew-symmetric, it is clear that zero is an eigenvalue for $[S_2]$, so there is a line of solution to $S_2(\bx)=0$, that pass through the origin and intercepts the sphere exactly in the points $p_\pm$. The other eigenvalues are a pure imaginary number and its complex conjugate. If we consider \bq\label{nu}\nu=\left(  \dfrac{a_{3,2}+b_{3,2}}{a_{2,1}+b_{2,1}},-\dfrac{a_{3,1}+b_{3,1}}{a_{2,1}+b_{2,1}},1  \right),\eq
for $a_{2,1}+b_{2,1}\neq 0$, then $S_2(\bx) \perp \nu$, for all $\bx$, so the trajectories of $S_2$ are closed (and planar).
\end{proof}

Now we back to equation \eqref{xh-esfera} and discuss the case where $Q$ (or $A$, by \cite{12}) is not negative definite.

In this case, following the classification of quadratic surfaces given in \cite{13}, we have that the solutions over the sphere of the equation $X\sigma_1(\bx)=0$ describes (see also Figure \ref{fig1}):\\
\noindent i) a circle, if the rank of $Q$ is one,\\
\noindent ii) two disjoint circles, if the rank of $Q$ is two,\\
\noindent iii) two intersecting (transversal) circles, if the rank of $Q$ is three.\\
\begin{figure}[h!]
\includegraphics[width=10cm]{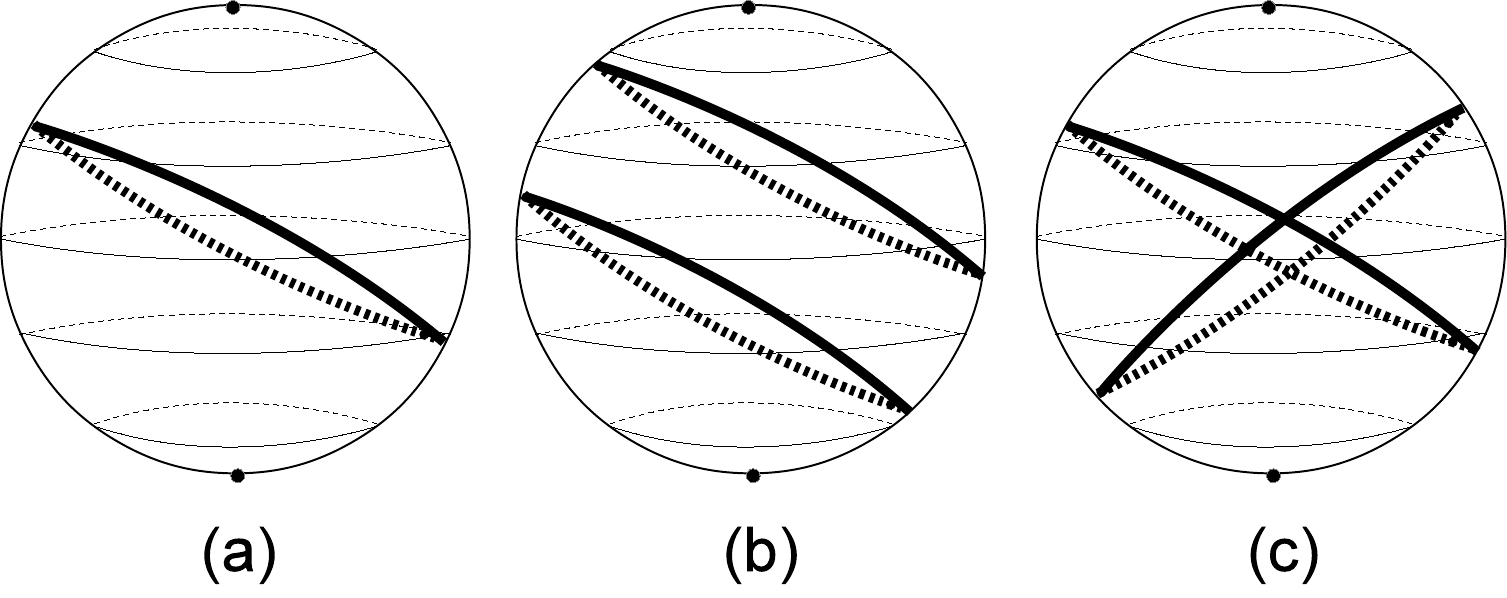}
\caption{Situation where (a) ${\rm rank}(Q)=1$, (b) ${\rm rank}(Q)=2$, (c) ${\rm rank}(Q)=3$. The bold lines are tangency circles.}
\label{fig1}
\end{figure}

Then a trajectory of the sliding vector field intersects the tangency set in at most 4 points, provided that this trajectory is not a tangency line. In case of finite intersection, recall that from the discussion in the end of Section \ref{nsvf} that we can also define the sliding vector field over the entire sphere.

Now we show that these circles of tangencies are never a trajectory of the sliding vector field $S_2$.

\begin{lema}\label{l4}Let $\gamma$ be a non trivial trajectory of $S_2$ over the sphere. Then $\gamma$ intercepts the set $\{\bx\in S^2;\, X\sigma_1(\bx)=0\}$ in at most 4 points.
\end{lema}
\begin{proof}The curve $\gamma$ is transversal to $\nu$ (given by \eqref{nu}). There are two options for the solutions of $X\sigma_1(\bx)=0,\, \langle \bx,\nu\rangle=0$: the empty set, or one line passing through the origin, or the union of two lines passing through the origin (if $a_{3,2}+b_{3,2}\neq 0$), parametrized by
\[
\gamma_{1,2}(t)=\left(  \dfrac{\xi_{1,2}b_{3,1}+\xi_{1,2} a_{3,1}-a_{2,1}-b_{2,1}}{a_{3,2}+b_{3,2}}\,t,\xi_{1,2} t,t     \right),
\]
where $\xi_{1},\xi_2$ are the roots of a two degree polynomial $\Xi(s)=0$ (given in the Appendix). If $\xi_1=\xi_2$ then we have just one line of solutions. The lines $\gamma_{1,2}$ intercepts the sphere in none, 2 or 4 points. In particular, the circles of tangencies never are a trajectory of the sliding vector field.
\end{proof}

So we can consider the (closure of the) sliding vector field on the whole sphere.

\begin{corolario}Lemma \ref{lema-ndm} is true even if $A$ is not negative definite.
\end{corolario}

\section{Proof of Theorem B}

We follow the same sequence of Section \ref{pta}.

Fix $\sigma_2(\bx)=(x_1^2+x_2^2+x_3^2+3)^2-16(x_1^2+x_2^2)$, so $\Sigma_2=\sigma_2^{-1}(\{0\}$ is the torus. Recall that $B$ is given by Lemma \ref{lema1}.

Then
\[
\begin{array}{lcl}
X\sigma_2(\bx)&=&q_2(\bx)+q_4(\bx)\\
\end{array}
\]
where
\[
\begin{array}{lcl}
q_2(\bx)&=&-20 a_{1,1}x_1^{2}+ \left( -20 a_{2,1}-20 a_{1,2} \right) x_2x_1-20 a_{2,2}x_2^{2}\\
&+& \left( -20 a_{1,3}+12 a_{3,1} \right) x_3x_1+ \left( 12 a_{3,2}-20 a_{2,3} \right) x_3x_2+12 a_{3,3}x_3^{2}
\end{array}
\]
and
\[
\begin{array}{lcl}
 q_4(\bx)&=&4 a_{1,1}x_1^{4}+ \left( 4 a_{2,1}+4 a_{1,2} \right) x_2x_1^{3}+ \left( 4 a_{1,1}+4 a_{2,2} \right) {x_2}^{2}x_1^{2}\\
 &+& \left( 4 a_
{2,1}+4 a_{1,2} \right) x_2^{3}x_1+4 a_{2,2}x_2^{4}+ \left( 4 a
_{1,3}+4 a_{3,1} \right) x_3x_1^{3}\\
&+& \left( 4 a_{2,3}+4 a_{3,2
} \right) x_3x_2{x_1}^{2}+ \left( 4 a_{1,3}+4 a_{3,1} \right) x_3{x_2}^{2
}x_1+ \left( 4 a_{2,3}+4 a_{3,2} \right) x_3{x_2}^{3}\\
&+& \left( 4 a_{1
,1}+4 a_{3,3} \right) {x_3}^{2}{x_1}^{2}+ \left( 4 a_{2,1}+4 a_{1
,2} \right) {x_3}^{2}x_2x_1+ \left( 4 a_{2,2}+4 a_{3,3} \right) {x_3}^{
2}{x_2}^{2}\\
&+& \left( 4 a_{1,3}+4 a_{3,1} \right) {x_3}^{3}x_1+ \left( 4
 a_{2,3}+4 a_{3,2} \right) {x_3}^{3}x_2+4 a_{3,3}{x_3}^{4}
\end{array}
\]

It is hard to determine the zeroes of $X\sigma_2$ for general functions $q_2,q_4$, So we work under the hypothesis
\[q_4(\bx)\equiv 0.\]

In this case, $X\sigma_2$ reduces to \[X\sigma_2(\bx)=q_2(\bx)=32a_{3,1}x_1x_3+32a_{3,2}x_2x_3.\]

\begin{lema}\label{l6}Under the hypothesis $q_4(\bx)\equiv 0$, the tangency set over $\Sigma_2$ is the union of four circles $C_1,C_2,C_3,C_4$, where $C_1,C_2$ are contained in the plane $x_3=0$ and $C_3,C_4$ are contained in the plane $a_{3,1}x_1+a_{3,2}x_2=0$ (see Figure \ref{toro}).
\end{lema}

\begin{figure}[h!]
\includegraphics[width=6cm]{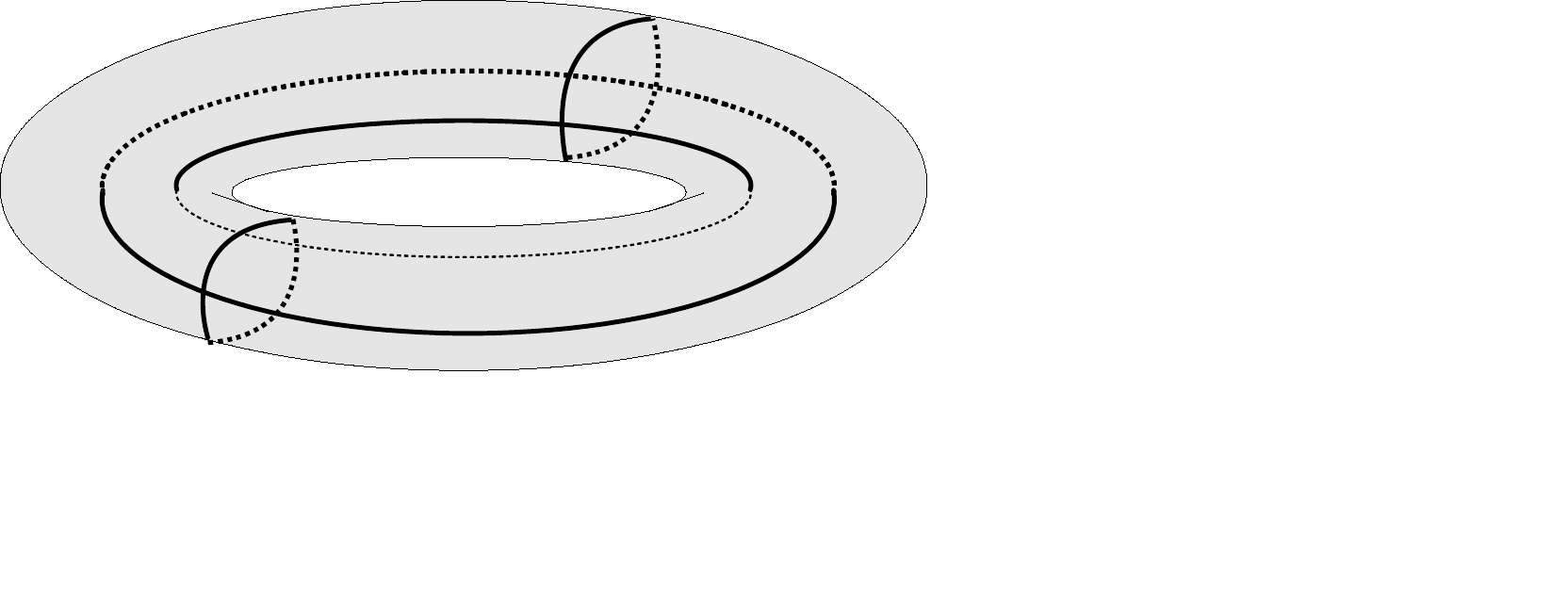}
\caption{Tangency lines over the torus.}
\label{toro}
\end{figure}

\begin{proof}Just determine the solutions of $q_2(\bx)=0$ and study its intersections with the torus $\Sigma_2$.
\end{proof}
Note that we have four intersection points between the circles in the statement of Lemma \ref{l6}, and these curves break the torus $\Sigma_2$ into four regions $R_1,R_2,R_3,R_4$, considering $R_1,R_2$ contained in the region $x_3>0$, and $R_3,R_4$ in the region $x_3<0$.

The expression of the sliding vector field is
\[S_2(\bx)=-\dfrac{1}{2}(a_{2,1}+b_{2,1})x_2\dfrac{\partial}{\partial x_1}+\dfrac12(a_{2,1}+b_{2,1})x_1\dfrac{\partial}{\partial x_1},\]
for $\bx\in R_1\cup R_2\cup R_3\cup R_4$.

The proof of the next two lemmas follow the same techniques of Lemmas \ref{lema-ndm} and \ref{l4}, and will be omitted. They prove Theorem A.

\begin{lema}The trajectories of the sliding vector field $S_2$ are transversal to the tangency circles $C_3,C_4$. So a trajectory of the sliding vector field can pass over these circles (crossing regions $R_1,R_2$ and $R_3,R_4$).
\end{lema}

\begin{obs} Note that if $\gamma_p$ is the trajectory of $S_2$ passing through $p\in \Sigma_2$, then  $\Sigma_2\setminus \left(\displaystyle \cup_{j=1}^4\cup_{p\in R_j} \gamma_{p}\right)=C_1\cup C_2$, that is, the curves $C_1,C_2$ are exactly the sets where we cannot define $S_2$.
\end{obs}

\begin{lema}The tangency circles $C_1,C_2$ are singular tangency points, that is, they behave like singular points of $S_2$.
\end{lema}

\appendix
\section{\label{xixi}Polynomial $\Xi(s)$.}
The polynomial $\Xi(s)$ is given by
\[
\Xi(s)=\Xi_2 {s}^{2}+ \Xi_1 s+\Xi_0,
\]
where
\[
\begin{array}{lcl}
\Xi_2&=&
2\,a_{{2,2}}b_{{3,2}}a_{{3,2}}+2\,a_{{1,1}}b_{{3,1}}a_{{3,1}}+
a_{{2,1}}b_{{3,2}}b_{{3,1}}+a_{{2,1}}b_{{3,2}}a_{{3,1}}\\
&+&a_{{2,1}}a_{{3
,2}}b_{{3,1}}+a_{{2,1}}a_{{3,2}}a_{{3,1}}+a_{{1,2}}b_{{3,2}}b_{{3,1}}+
a_{{1,2}}b_{{3,2}}a_{{3,1}}\\
&+&a_{{1,2}}a_{{3,2}}b_{{3,1}}+a_{{1,2}}a_{{3
,2}}a_{{3,1}}+a_{{1,1}}{b_{{3,1}}}^{2}+a_{{1,1}}{a_{{3,1}}}^{2}+a_{{2,
2}}{b_{{3,2}}}^{2}+a_{{2,2}}{a_{{3,2}}}^{2}\\
\Xi_1&=&{a
_{{3,2}}}^{3}-2\,b_{{2,1}}b_{{3,1}}a_{{1,1}}-2\,b_{{2,1}}a_{{3,1}}a_{{
1,1}}-b_{{2,1}}b_{{3,2}}a_{{2,1}}-b_{{2,1}}a_{{3,2}}a_{{2,1}}-b_{{2,1}
}b_{{3,2}}a_{{1,2}}\\
&-&b_{{2,1}}a_{{3,2}}a_{{1,2}}+b_{{3,1}}b_{{3,2}}a_{{
1,3}}+a_{{3,1}}b_{{3,2}}a_{{1,3}}+b_{{3,1}}a_{{3,2}}a_{{1,3}}+a_{{3,1}
}a_{{3,2}}a_{{1,3}}\\
&-&2\,a_{{1,1}}b_{{3,1}}a_{{2,1}}-2\,a_{{1,1}}a_{{3,1
}}a_{{2,1}}-a_{{1,2}}b_{{3,2}}a_{{2,1}}-a_{{1,2}}a_{{3,2}}a_{{2,1}}+a_
{{3,1}}b_{{3,2}}b_{{3,1}}\\
&+&a_{{3,1}}a_{{3,2}}b_{{3,1}}+2\,a_{{3,2}}b_{{
3,2}}a_{{2,3}}-{a_{{2,1}}}^{2}b_{{3,2}}-{a_{{2,1}}}^{2}a_{{3,2}}+{a_{{
3,1}}}^{2}b_{{3,2}}+{a_{{3,1}}}^{2}a_{{3,2}}\\
&+&{b_{{3,2}}}^{2}a_{{3,2}}+
2\,b_{{3,2}}{a_{{3,2}}}^{2}+{b_{{3,2}}}^{2}a_{{2,3}}+{a_{{3,2}}}^{2}a_
{{2,3}}\\
\Xi_0&=&-a_{{1,3}}a_{{3,2}}a_{{2,1}}+2\,a_{{3,3}}b_{{3,2}}a_{
{3,2}}+2\,a_{{1,1}}a_{{2,1}}b_{{2,1}}-a_{{2,1}}b_{{3,2}}a_{{3,1}}-a_{{
3,1}}b_{{3,2}}b_{{2,1}}\\
&-&a_{{2,1}}a_{{3,2}}a_{{3,1}}-a_{{3,1}}a_{{3,2}}
b_{{2,1}}-a_{{1,3}}b_{{3,2}}a_{{2,1}}-a_{{1,3}}b_{{3,2}}b_{{2,1}}+a_{{
1,1}}{a_{{2,1}}}^{2}\\
&+&a_{{1,1}}{b_{{2,1}}}^{2}+a_{{3,3}}{b_{{3,2}}}^{2}
+a_{{3,3}}{a_{{3,2}}}^{2}-a_{{1,3}}a_{{3,2}}b_{{2,1}}
\end{array}
\]

\end{document}